\documentclass[11pt]{amsart}

\usepackage{amsfonts, amsmath, amscd}
\usepackage[psamsfonts]{amssymb}

\usepackage{amssymb}

\usepackage[usenames]{color}

\newtheorem{theorem}{Theorem}[section]
\newtheorem{lemma}[theorem]{Lemma}
\newtheorem{corollary}[theorem]{Corollary}

\newtheorem{remark}[theorem]{Remark}
\newtheorem{proposition}[theorem]{Proposition}
\newtheorem{definition}[theorem]{Definition}
\newtheorem{example}[theorem]{Example}

\numberwithin{equation}{section}

\newcommand{\CC}{C_k}
\newcommand{\NN}{\mathbb{N}}

\newcommand{\w}{\omega}

\newcommand{\KK}{\mathcal{K}}
\newcommand{\RR}{\mathbb{R}}
\newcommand{\VV}{\mathbb{V}}

\newcommand{\Nn}{\mathcal{N}}

\newcommand{\IR}{\mathbb{R}}

\newcommand{\xxx}{\mathbf{x}}
\newcommand{\yyy}{\mathbf{y}}

\newcommand{\Pp}{\mathfrak{P}}

\newcommand{\e}{\varepsilon}

\renewcommand{\phi}{\varphi}

\newcommand{\U}{\mathcal U}

\newcommand{\supp}{\mathrm{supp}}
\newcommand{\conv}{\mathrm{conv}}
\newcommand{\spn}{\mathrm{span}}

\newcommand{\Tt}{\mathfrak{T}}

\input xy
\xyoption{all}

\title[On the Ascoli property for locally convex spaces]{On the Ascoli property for locally convex spaces}

\author[S. S.~Gabriyelyan]{Saak Gabriyelyan}
\address{Department of Mathematics, Ben-Gurion University of the
Negev, Beer-Sheva, P.O. 653, Israel}
\email{saak@math.bgu.ac.il}



\begin{document}

\begin{abstract}
We characterize Ascoli spaces by showing that a Tychonoff space $X$ is Ascoli iff the canonical map from the free locally convex space $L(X)$ over $X$ into $C_k\big(C_k(X)\big)$ is an embedding of locally convex spaces. We prove that an uncountable direct sum of non-trivial locally convex spaces is not Ascoli.
If a $c_0$-barrelled space $X$ is weakly Ascoli, then $X$ is linearly isomorphic to a dense subspace of $\mathbb{R}^\Gamma$ for some $\Gamma$. Consequently, a Fr\'{e}chet space $E$ is weakly Ascoli iff $E=\mathbb{R}^N$ for some $N\leq\omega$. If $X$ is a $\mu$-space and a $k$-space (for example, metrizable), then $C_k(X)$ is weakly Ascoli iff $X$ is discrete. We prove that the weak* dual space of a Banach space $E$ is Ascoli iff $E$ is finite-dimensional.
\end{abstract}

\maketitle



\section{Introduction. }


Topological properties of  the space $C(X)$ of all continuous real-valued functions  on  a Tychonoff space $X$ endowed with the pointwise topology or the compact-open topology, which we denote by  $C_p(X)$ and $\CC(X)$, respectively, are of great importance and have been intensively studied from many years (see, for example, \cite{Arhangel,Jar,kak,mcoy} and references therein). Let us mention metrizability, the Fr\'{e}chet--Urysohn property, sequentiality, the $k$-space property, the $k_\IR$-space property and countable
tightness. It is easy to see that $C_p(X)$ is metrizable if and only if $X$ is countable. Pytkeev, Gerlitz and Nagy (see \S 3 of \cite{Arhangel}) characterized  spaces $X$ for which $C_p(X)$ is Fr\'{e}chet--Urysohn, sequential or a $k$-space (these properties coincide for the spaces $C_p(X)$). McCoy proved in \cite{mcoy1} that for a first countable paracompact $X$ the space $\CC(X)$ is a $k$-space if and only if $X$ is hemicompact, so $\CC(X)$ is metrizable.
Being motivated by the classic Ascoli theorem we introduced in \cite{BG} a new class of topological spaces, namely, the class of Ascoli spaces. A Tychonoff space $X$  is {\em Ascoli} if every compact subset of $\CC(X)$  is evenly continuous.
By Ascoli's theorem \cite[Theorem 3.4.20]{Eng}, each $k$-space is Ascoli, and Noble \cite{Noble} proved that any $k_\IR$-space is Ascoli.
So we have the following diagram
\[
\xymatrix{
\mbox{metric} \ar@{=>}[r] & {\mbox{Fr\'{e}chet--}\atop\mbox{Urysohn}} \ar@{=>}[r] & \mbox{sequential} \ar@{=>}[r] &  \mbox{$k$-space} \ar@{=>}[r] &  \mbox{$k_\IR$-space} \ar@{=>}[r] &  {\mbox{Ascoli}\atop\mbox{space}}, }
\]
and none of these implications is reversible. The Ascoli property for function spaces has been studied recently  in \cite{Banakh-Survey,BG,Gabr-C2,GGKZ,GGKZ-2,GKP}.
Let us mention the following
\begin{theorem}  \label{t:Cp-Ascoli-lc}
{\rm (i) (\cite{GGKZ-2})} If $X$ is a locally compact space, then $C_p(X)$ is an Ascoli space if and only if $X$ is scattered. In particular, for every ordinal $\kappa$ the space $C_p(\kappa)$ is Ascoli.

\smallskip
{\rm (ii) (\cite{GKP})} If $X$ is metrizable, then $\CC(X)$ is Ascoli if and only if $X$ is locally compact.
\end{theorem}

It is well-known (see \cite{Arhangel}) that a topological space $X$ is Tychonoff if and only if the canonical valuation map $X\to C_p(C_p(X))$ is an embedding. Replacing the pointwise topology by the compact-open topology we obtain a characterization of Ascoli spaces given in \cite{BG}: $X$ is Ascoli if and only if  the canonical map $X\to \CC(\CC(X))$ is an embedding. Below we obtain a `locally convex' characterization of  Ascoli spaces  using the notion of free locally convex space introduced by Markov \cite{Mar} (for all relevant definitions see Section \ref{sec-1}). 
\begin{theorem} \label{t:Ascoli-free-map}
A Tychonoff space $X$ is  Ascoli if and only if  the canonical map $\Delta_X: L(X)\to \CC\big(\CC(X)\big)$  is an embedding of locally convex spaces.
\end{theorem}

If $X=D$ is a countable infinite discrete space, then the space $L(D)$ coincides with the direct sum $\phi$ of countably many copies of $\IR$ endowed with the box topology. Note that $\phi$ is a sequential non-Fr\'{e}chet--Urysohn space, in particular, $\phi$ is Ascoli. However, if $D$ is an {\em uncountable} discrete space the situation changes: $L(D)$ is not an Ascoli space. This result is stated in \cite{Banakh-Survey}, we give an elementary direct proof of a more general assertion, see Theorem \ref{tFreeSpace-Ascoli} below.

It is well-known that the {\em locally convex} properties to be a barrelled, Mackey or bornological space are preserved under taking locally convex direct sums. These facts motivate the following question: When the locally convex direct sum $E$ of a family of non-trivial Ascoli locally convex spaces is  an Ascoli space? We proved in \cite{Gab-LF} that the locally convex direct sum $E$ of a {\em sequence} $\{ E_n\}_{n\in\w}$ of nontrivial {\em metrizable} locally convex spaces is  an Ascoli space if and only if all the $E_n$ are finite-dimensional, so $E=\phi$. In the next theorem we show that {\em uncountable} locally convex direct sums of locally convex spaces are never Ascoli. Moreover, we show that the operation of taking  uncountable locally convex direct sums of locally convex spaces destroys also sequential-type and generalized metric type {\em topological} properties.

\begin{theorem} \label{t:Ascoli-directsum-LCS}
Let $E:=\bigoplus_{i\in\kappa} E_i$ be the direct sum of an uncountable family $\{ E_i\}_{i\in\kappa}$  of non-trivial locally convex spaces  and let $\tau$ be either  the locally convex direct sum topology $\mathcal{T}_{lc}$ or the box topology $\mathcal{T}_b$ on   $E$. Then:
\begin{enumerate}
\item[{\rm (i)}] $(E,\tau)$ is not an Ascoli space;
\item[{\rm (ii)}] $(E,\tau)$   has uncountable tightness;
\item[{\rm (iii)}] $(E,\tau)$ has uncountable $cs^\ast$-character and uncountable $cn$-character.
\end{enumerate}
\end{theorem}

A systematic study of topological properies of the weak topology of Banach spaces was proposed by Corson in \cite{Corson}. Schl\"{u}chtermann and Wheeler \cite{S-W} showed that an infinite-dimensional Banach space in the weak topology is never a $k$-space. Let us also recall that the famous Kaplansky theorem states that a normed  space $E$ in the weak topology has countable tightness; for a generalization of this result see \cite{GKZ}.  We shall say that a locally convex space $E$  is {\em weakly Ascoli} if $E$ endowed with the weak topology $\sigma(E,E')$ is an Ascoli space (where $E'$ denotes the topological dual space of $E$). In \cite{GKP} we obtained a characterization of weakly Ascoli Banach spaces.
\begin{theorem}[\cite{GKP}] \label{t:Ascoli-weak-Banach}
A Banach space $E$ is weakly Ascoli if and only if it is finite-dimensional.
\end{theorem}

Taking into account that every Banach space $E$ is Baire and hence $E$ is barrelled and bornological, Theorem \ref{t:Ascoli-weak-Banach} motivates the following question: Which barrelled (bornological, Baire etc.) locally convex spaces are weakly Ascoli?
In Theorem \ref{t:Ascoli-c0-bar} below we obtain a necessary condition to be weakly Ascoli in a more general class of $c_0$-barrelled spaces.
Recall that a locally convex space $X$ is called {\em $c_0$-barrelled} if every $\sigma(X',X)$-null sequence in $X'$ is equicontinuous. Every barrelled space and hence every Fr\'{e}chet space is $c_0$-barrelled.
\begin{theorem} \label{t:Ascoli-c0-bar}
If a $c_0$-barrelled space $X$ is weakly Ascoli, then:
\begin{enumerate}
\item[{\rm (i)}] every $\sigma(X',X)$-bounded subset of $X'$ is finite-dimensional;
\item[{\rm (ii)}] the topology $\tau$ of $X$ coincides with the weak topology;
\item[{\rm (iii)}] $X$ is linearly isomorphic to a dense subspace of $\IR^\Gamma$, where $\Gamma$ is a Hamel base of $X'$.
\end{enumerate}
\end{theorem}

If a $c_0$-barrelled space is  complete we obtain
\begin{corollary} \label{c:Ascoli-dual-Frechet}
A complete $c_0$-barrelled space $X$ is weakly Ascoli  if and only if $X=\IR^\Gamma$ for some $\Gamma$. In particular, a Fr\'{e}chet space $E$ is weakly Ascoli if and only if $E=\IR^N$ for some $N\leq\w$.
\end{corollary}
Note that the last assertion of Corollary \ref{c:Ascoli-dual-Frechet} generalizes Theorem \ref{t:Ascoli-weak-Banach} and Theorem 1.6 of \cite{GKP} by removing the condition on a Fr\'{e}chet space $E$ to be a quojection. The class of $c_0$-barrelled spaces is much wider than the class of barrelled spaces (see Chapter 12 in \cite{Jar}), but in the class of weakly Ascoli locally convex spaces these two classes coincide.
\begin{corollary} \label{c:weak-bar-c0-bar}
Let $X$ be a weakly Ascoli locally convex space. Then $X$ is $c_0$-barrelled if and only if $X$ is barrelled.
\end{corollary}
We prove this corollary using a characterization of locally convex spaces which are barrelled in the weak topology, see Theorem \ref{t:weak-barrelled}.

One can naturally ask a more concrete question: Which barrelled (bornological, Baire etc.) spaces $C_p(X)$ and $\CC(X)$ are weakly Ascoli? The next theorem shows that the answers to this question for $C_p(X)$ and $\CC(X)$ are the same.
\begin{theorem} \label{t:Ck-bar-weakly-Ascoli}
For a Tychonoff space $X$ the following assertions are equivalent:
\begin{enumerate}
\item[{\rm (i)}] $\CC(X)$ is a barrelled weakly Ascoli space;
\item[{\rm (ii)}] $C_p(X)$ is a barrelled Ascoli space.
\end{enumerate}
\end{theorem}

Theorem \ref{t:Ck-bar-weakly-Ascoli}, the Nachbin--Shirota theorem and the Buchwalter--Schmets theorem imply
\begin{corollary} \label{c:Ck-weakly-Acoli}
If $X$ is a $\mu$-space and a $k$-space (for example, metrizable), then $\CC(X)$ is weakly Ascoli if and only if $X$ is discrete.
\end{corollary}

In the  proof of Theorem \ref{t:Ascoli-weak-Banach} we essentially used the fact (discovered in \cite{GKZ}) that every Banach space $E$ in the weak topology has countable fan tightness. This property introduced in \cite{Arhangel} is stronger than the countable tightness. Since the weak${}^\ast$ dual of a Banach space may not have countable tightness, this fact cannot be used to prove the following theorem which naturally complements Theorem \ref{t:Ascoli-weak-Banach}.
\begin{theorem} \label{t:Ascoli-weak-dual}
The weak${}^\ast$ dual space of a Banach space $E$ is Ascoli if and only if $E$ is finite-dimensional.
\end{theorem}

We can now describe the content of the paper. In Section \ref{sec-1} we discuss different notions of even continuity and show that a Tychonoff space $X$ is Ascoli if and only if each compact subset of $\CC(X)$ is equicontinuous (see Lemma \ref{l:Ascoli-free-evenly}(iv)). Theorem \ref{t:Ascoli-free-map} is proved also in Section \ref{sec-1}, and its application is given. In   Section \ref{sec-2} we prove Theorem \ref{t:Ascoli-directsum-LCS}, and  Theorems \ref{t:Ascoli-c0-bar}, \ref{t:Ck-bar-weakly-Ascoli} and \ref{t:Ascoli-weak-dual}  are proved in the last section.


\section{A characterization of Ascoli spaces} \label{sec-1}


All spaces in the article are assumed to be Tychonoff. The closure of a subset $A$ of a topological space $(X,\tau)$ we denote by $\overline{A}$ or $\overline{A}^{\, \tau}$. Recall that  a subset $A$ of a topological space $X$ is called {\em functionally bounded} in $X$ if every continuous function on $X$ is bounded on $A$.

For Tychonoff spaces $X$ and $Y$ we denote by $C(X,Y)$ the space of all continuous functions from $X$ to $Y$.
Let $\Tt$ be a {\em directed family} of functionally bounded subsets of $X$ (i.e., if $A,B\in \Tt$ then there is $C\in\Tt$ such that $A\cup B\subseteq C$) containing all finite subsets. Denote by $\tau_\Tt$ the set-open topology on $C(X,Y)$ generated by $\Tt$. The subbase of $\tau_\Tt$ consists of the sets
\[
[A;U]=\{ f\in C(X,Y): \; f(A)\subseteq U\},
\]
where $A\in \Tt$ and $U$ is an open subset of $Y$. The space $C(X,Y)$ with the topology $\tau_\Tt$ is denoted by $C_\Tt(X,Y)$. If $\Tt$ is the family of all finite subsets or the family of all compact subsets of $X$ we obtain the pointwise topology $\tau_p$ and the compact-open topology $\tau_k$, respectively, and write $C_p(X,Y)$ and $\CC(X,Y)$. Clearly, $\tau_p \leq \tau_\Tt$. Denote by
\[
\psi: X\times C_\Tt(X,Y) \to Y, \quad \psi(x,f):=f(x),
 \]
the valuation map.

Recall that a subset $Z$ of $C(X,Y)$ is called {\em evenly continuous} if for each $x\in X$, each $y\in Y$ and each neighborhood $O_y$ of $y$ there is a neighborhood $O_x$ of $x$ and a neighborhood $V_y$ of $y$ such that $f(O_x)\subseteq O_y$ whenever $f(x)\in V_y$.
Below we introduce a weaker notion than even continuity.

\begin{definition} {\em
A  subset $\KK$ of $C(X,Y)$ is called {\em $\tau_\Tt$-evenly continuous} or {\em evenly continuous in $C_\Tt(X,Y)$} if the restriction of the valuation map $\psi$ onto $X\times \KK$ is jointly continuous, i.e.,  for any $x\in X$, each $f\in\KK$  and every neighborhood $O_{f(x)}$ of $f(x)$ there exists a $\tau_\Tt$-neighborhood $U_f\subseteq \KK$ of $f$ and a neighborhood $O_x\subseteq X$ of $x$ such that
\[
\{g(y):g\in U_f,\;y\in O_x\}\subseteq O_{f(x)}.
\] }
\end{definition}

So the notion of  $\tau_\Tt$-even continuity depends on the topology  $\tau_\Tt$ on $C(X,Y)$.
Lemma 3.4.18 of \cite{Eng} states that every evenly continuous subset of $C(X,Y)$ is $\tau_p$-evenly continuous. Since $\tau_p \leq \tau_k$, every $\tau_p$-evenly continuous subset of $C(X,Y)$ is also $\tau_k$-evenly continuous.

If $Y=\IR$ we write simply $C(X)$, $C_\Tt(X)$, $C_p(X)$ or $\CC(X)$, respectively. Then the family
\[
\big\{ [A,\e]: \, A\in \Tt, \, \e>0\big\}, \quad \mbox{ where } [A,\e]:=\{ f\in C(X): f(A)\subseteq (-\e,\e)\},
\]
is a base of the locally convex topology $\tau_\Tt$.

Recall that a subset $Z$ of $C(X)$ is called {\em equicontinuous at a point $x\in X$}  if for every $\e >0$ there is a neighborhood $O_x$ of $x$ such that $|f(y)-f(x)|<\e$ for every $y\in O_x$ and each $f\in Z$; $Z$ is  {\em equicontinuous } if it is equicontinuous at each point $x\in X$. It is clear that every  equicontinuous subset of $C(X)$ is evenly continuous. This fact and the discussion above imply the following diagram
\[
\xymatrix{
\mbox{equicontinuity} \ar@{=>}[r] & {\mbox{even}\atop\mbox{continuity}} \ar@{=>}[r] &  {\mbox{$\tau_p$-even}\atop\mbox{continuity}} \ar@{=>}[r] & {\mbox{$\tau_k$-even}\atop\mbox{continuity}} . }
\]
It can be shown that none of these implications is reversible even for a Banach space $C(X)$.

We defined in \cite{BG} a Tychonoff space $X$ to be {\em Ascoli} if every compact subset $\KK\in \CC(X)$ is $\tau_k$-evenly continuous. However in this definition the $\tau_k$-even continuity can be replaced by equicontinuity as (iv) of the following lemma shows.
\begin{lemma} \label{l:Ascoli-free-evenly}
Let $X$ be  a Tychonoff space and let $\Tt$ and $\Tt'$ be directed families of functionally bounded subsets of $X$ containing all finite subsets of $X$. Then:
\begin{enumerate}
\item[{\rm (i)}] every equicontinuous subset $\KK$ of $C(X)$ is $\tau_\Tt$-evenly continuous;
\item[{\rm (ii)}] every $\tau_\Tt$-evenly continuous and $\tau_\Tt$-compact subset $\KK$ of $C(X)$ is equicontinuous;
\item[{\rm (iii)}] if $\Tt\leq \Tt'$, then every  $\tau_\Tt$-evenly continuous subset of $C(X)$ is $\tau_{\Tt'}$-evenly continuous;
\item[{\rm (iv)}] a compact subset $\KK$ of $\CC(X)$ is equicontinuous  if and only if  $\KK$  is evenly continuous  if and only if $\KK$ is $\tau_p$-evenly continuous if and only if  $\KK$ is $\tau_k$-evenly continuous;
\item[{\rm (v)}] a compact subset $\KK$ of $C_p(X)$ is equicontinuous  if and only if  $\KK$  is evenly continuous  if and only if $\KK$ is $\tau_p$-evenly continuous.
\item[{\rm (vi)}]  the canonical map $\delta: X\to \CC\big(C_\Tt(X)\big)$, $\delta(x)(f):=f(x)$, is continuous if and only if every $\tau_\Tt$-compact subset of $C(X)$ is equicontinuous.
\end{enumerate}
\end{lemma}
\begin{proof} 
(i) Fix $(x,f)\in X\times \KK$ and $\e>0$. Choose a neighborhood $O_x$ of $x$ such that
\[
|g(y)-g(x)|<\e/2, \quad \forall y\in O_x, \; \forall g\in \KK.
\]
Set $A:=\{ x\}$ and $U_f:= (f+[A,\e/2])\cap \KK$. Now if $h\in  U_f$ and $y\in O_x$, we obtain
\[
| \psi(y,h)-\psi(x,f)|=|h(y)-f(x)| \leq |h(y)-h(x)| + |h(x)-f(x)| <\e/2 +\e/2 =\e.
\]
Thus $\psi$ is continuous at $(x,f)$, and hence $\KK$ is $\tau_\Tt$-evenly continuous.

(ii) Let $x\in X$. Fix $\e>0$. Since $\KK$ is $\tau_\Tt$-evenly continuous, for every $f\in\KK$ and $\e/2$-neighborhood $O_{f(x)}$ of $f(x)$ there exists an open $\tau_\Tt$-neighborhood $U_f\subseteq \KK$ of $f$ and a neighborhood $O_{x,f}\subseteq X$ of $x$ such that
\begin{equation} \label{equ:lemma-Ascoli-equicon}
|g(y)-f(x)|<\e/2, \quad \forall y\in O_{x,f}, \; \forall g\in U_f.
\end{equation}
Since $\KK$ is $\tau_\Tt$-compact, there are $f_1,\dots, f_n\in\KK$ such that $\KK=\bigcup_{i\leq n} U_{f_i}$. Set $O_x := \bigcap_{i\leq n} O_{x,f_i}$. Now, let $y\in O_x$ and $g\in\KK$. Choose $1\leq i\leq n$ such that $g\in U_{f_i}$. Since $y\in O_x \subseteq O_{x,f_i}$ the inequality (\ref{equ:lemma-Ascoli-equicon}) implies
\[
|g(y)-g(x)|\leq |g(y)-f_i(x)| + |f_i(x)-g(x)|<\e.
\]
So $\KK$ is equicontinuous at $x$. Since $x$ is arbitrary, $\KK$ is equicontinuous.

(iii) is trivial, and (iv),(v) follow from (ii) and the diagram before the lemma.

(vi) We shall write $\delta(x):=\delta_x$ for $x\in X$. Assume that $\delta$ is continuous. Let $\KK$ be a compact subset of $C_\Tt(X)$. We have to check that $\KK$ is equicontinuous. To this end, by (ii), it is sufficient to show that $\KK$ is $\tau_\Tt$-evenly continuous. Fix $x\in X$, $f\in\KK$  and an open neighborhood $O_{f(x)}$ of $f(x)$. Choose an open neighborhood $\widetilde{O}_{f(x)}$ of $f(x)$ such that $\mathrm{cl} \big( \widetilde{O}_{f(x)}\big) \subseteq O_{f(x)}$. Let $U_f := \big[ \{ x\}; \widetilde{O}_{f(x)}\big] \cap\KK$ (recall that $\{ x\}\in\Tt$) and set $\mathcal{C}:= \mathrm{cl}_\KK (U_f)$. Then for every $g\in\mathcal{C}$ we have
\[
\delta_x(g)=g(x) \in \mathrm{cl} \big( \widetilde{O}_{f(x)}\big) \subseteq O_{f(x)},
\]
and therefore $\delta_x \in [\mathcal{C}; O_{f(x)}]$. Since $\delta$ is continuous, there is a neighborhood $O_x$ of $x$ such that $\delta(O_x)\subseteq \big[\mathcal{C}; O_{f(x)}\big]$. So for every $y\in O_x$ and each $g\in U_f \subseteq \mathcal{C}$ we have $g(y)=\delta_y(g)\in O_{f(x)}$, which means that $\KK$ is $\tau_\Tt$-evenly continuous.
\smallskip

Conversely, assume that every $\tau_\Tt$-compact subset of $C(X)$ is equicontinuous. We have to show that $\delta$ is continuous at each point $x_0\in X$. Fix a basic neighborhood
\[
[\KK;\e]\subseteq\CC\big( C_\Tt(X)\big)
\]
of zero with $\KK\subseteq C_\Tt(X)$ compact and $\e>0$. Since $\KK$ is equicontinuous there is a neighborhood $O_{x_0}$ of $x_0$ such that
\begin{equation} \label{equ:lemma-Ascoli-equicon-2}
|f(x)-f(x_0)|<\e, \quad \forall x\in O_{x_0}, \; \forall f\in \KK.
\end{equation}
Then, for every $ x\in O_{x_0}$ and each $f\in\KK$, (\ref{equ:lemma-Ascoli-equicon-2}) implies
\[
|\delta_x(f) -\delta_{x_0}(f)|=|f(x)-f(x_0)|<\e, \quad \forall f\in\KK.
\]
Therefore $\delta_x \in \delta_{x_0}+[\KK;\e]$. This means that $\delta_{x}$ is continuous at $x_0$.
\end{proof}

Recall that a subset $A$ of $C(X)$ is called {\em pointwise bounded} if the set $\{ f(x):f\in A\}$ has compact closure in $\IR$ for every $x\in X$.
We shall use the following fact proved in the ``if'' part of the Ascoli theorem \cite[Theorem 3.4.20]{Eng}.
\begin{proposition} \label{p:Ascoli-Free-Ck}
Let $X$ be a Tychonoff space and $\KK$ be an evenly continuous pointwise bounded subset of $C(X)$. Then the $\tau_p$-closure ${\bar A}$ of $A$ is $\tau_k$-compact and evenly continuous.  Moreover, $\tau_k|_{{\bar A}} =\tau_p|_{{\bar A}}$.
\end{proposition}

Recall that the {\em  free locally convex space} $L(X)$ over a Tychonoff space $X$ is a pair consisting of a locally convex space $L(X)$ and  a continuous mapping $i: X\to L(X)$ such that every  continuous mapping $f$ from $X$ to a locally convex space $E$ gives rise to a unique continuous linear operator ${\bar f}: L(X) \to E$  with $f={\bar f} \circ i$. The free locally convex space $L(X)$  always exists and is  unique.
The set $X$ forms a Hamel basis for $L(X)$, and  the mapping $i$ is a topological embedding \cite{Rai,Usp2}.

Denote by $\delta_X:  X\mapsto \CC(\CC(X))$, $\delta_X (x) (f):=f(x)$, the canonical valuation map. Taking into account the definition of $L(X)$, the map $\delta_X$ can be extended to the canonical linear monomorphism $\Delta_X: L(X)\to \CC\big(\CC(X)\big)$  defined by the assignment
\[
\Delta_X \big( a_1 x_1 +\cdots + a_n x_n\big)(f) :=a_1 f(x_1) +\cdots + a_n f(x_n),
\]
where $n\in\NN$, $x_1,\dots, x_n\in X$, $a_1,\dots,a_n\in\IR$ and $f\in \CC(X)$. If $X$ is a $k$-space, Flood~\cite{Flo2} and Uspenski\u{\i}~\cite{Usp} proved that $\Delta_X$ is an embedding  of locally convex spaces. So Theorem \ref{t:Ascoli-free-map} generalizes this result.
Now we are ready to prove Theorem \ref{t:Ascoli-free-map}.

\begin{proof}[Proof of Theorem \ref{t:Ascoli-free-map}]
{\em Necessity.}
In \cite{Rai}  Raikov  showed that the topology $\pmb{\nu}_X$ of $L(X)$ is the topology of uniform convergence on equicontinuous pointwise bounded subsets $A$ of $C(X)$.  Lemma \ref{l:Ascoli-free-evenly} and Proposition \ref{p:Ascoli-Free-Ck} imply that the closure $\overline{A}^{\, \tau_{k}}$ of $A$ is $\tau_k$-compact and $\tau_k$-evenly continuous. Conversely, if $A$ is  $\tau_k$-evenly continuous and $\tau_k$-compact, then  $A$ is equicontinuous by Lemma \ref{l:Ascoli-free-evenly}. Clearly, $A$ is $\tau_p$-compact, and hence $\{ f(x):f\in A\}$ is compact in $\IR$ for every $x\in X$. Therefore $A$ is also pointwise bounded. Hence the topology $\pmb{\nu}_X$ coincides with  the topology of uniform convergence on compact equicontinuous subsets of $\CC(X)$. Since the space $X$ is Ascoli, every compact subset of $\CC(X)$ is equicontinuous  by Lemma \ref{l:Ascoli-free-evenly}. Taking into account that the canonical map $\Delta_X$ is injective we obtain that $\pmb{\nu}_X$ coincides with the compact-open topology inherited from $\CC\big(\CC(X)\big)$. Thus $\Delta_X$ is an embedding.

{\em Sufficiency.} If $\Delta_X$ is an embedding, then the canonical map $\delta_X =\Delta_X |_X$ is an embedding as well (recall that $X$ is a subspace of $L(X)$). Thus $X$ is an Ascoli space by Lemma \ref{l:Ascoli-free-evenly}. 
\end{proof}

Below we give an application of Theorem \ref{t:Ascoli-free-map}. First we recall some definitions.

Following Markov \cite{Mar}, a topological group  $A(X)$ is called {\em  the (Markov) free abelian topological  group} over  $X$ if $A(X)$ satisfies the following conditions: (i) there is a continuous mapping  $i: X\to A(X)$ such that $i(X)$ algebraically generates $A(X)$, and (ii) if $f: X\to G$ is a continuous mapping to an abelian topological  group $G$, then there exists a continuous homomorphism ${\bar f}: A(X) \to G$ such that $f={\bar f} \circ i$. The topological  group $A(X)$ always exists and is essentially unique, the mapping $i$ is a topological embedding \cite{Mar}. Note also that the identity map $id_X :X\to X$ extends to a canonical homomorphism $id_{A(X)}: A(X)\to L(X)$ which is an embedding of topological groups \cite{Tkac, Usp2}.

Following Michael \cite{Mich}, a topological space $X$ is called an \emph{$\aleph_0$-space}, if $X$ is a regular space with a countable $k$-network (a family $\mathcal{N}$ of subsets of $X$ is called a \emph{$k$-network} in $X$ if, whenever $K\subseteq U$ with $K$ compact and $U$ open in $X$, then $K\subseteq \bigcup \mathcal{F}\subseteq U$ for some finite family $\mathcal{F} \subseteq\mathcal{N}$). A topological space $X$ is called \emph{cosmic} \cite{Mich}, if $X$ is a regular space with a countable network (a family $\mathcal{N}$ of subsets of $X$ is called a \emph{network} in $X$ if, whenever $x\in U$ with $U$ open in $X$, then $x\in N\subseteq U$ for some $N \in\mathcal{N}$).
Following Banakh \cite{Banakh}, a  topological space $X$ is called  a {\em $\Pp_0$-space} if $X$ has a countable Pytkeev network (a family $\mathcal{N}$ of subsets of a topological space $X$ is called a  {\em Pytkeev network} if $\Nn$ is a network in $X$ and  for every point $x\in X$ and every open set $U\subseteq X$ and a set $A$ accumulating at $x$ there is a set $N\in\Nn$ such that $N\subseteq U$ and $N\cap A$ is infinite). Any $\Pp_0$-space is an $\aleph_0$-space, see \cite{Banakh}.

It is known (see \cite{GM}) that, for a Tychonoff space $X$, the space $L(X)$ is cosmic if and only if the group $A(X)$ is cosmic  if and only if the space $X$ is cosmic. For $k$-spaces the next corollary is proved in \cite{Gab-MSJ}.
\begin{corollary} \label{c:Free-Aleph}
Let $X$ be an Ascoli space. Then:
\begin{enumerate}
\item[{\rm (i)}]  $L(X)$  is an $\aleph_0$-space if and only if $A(X)$  is an $\aleph_0$-space if and only if $X$  is an  $\aleph_0$-space;
\item[{\rm (ii)}]  $L(X)$  is a $\Pp_0$-space if and only if $A(X)$  is a $\Pp_0$-space if and only if $X$  is a  $\Pp_0$-space.
\end{enumerate}
\end{corollary}

\begin{proof}
If $L(X)$  is an $\aleph_0$-space (a $\Pp_0$-space), then so is $A(X)$ as a subspace of $L(X)$. Analogously, if $A(X)$  is an $\aleph_0$-space (a $\Pp_0$-space), then so is $X$ as a subspace of $A(X)$.  If $X$ is an $\aleph_0$-space (a $\Pp_0$-space), then so are $\CC(X)$ and  $\CC(\CC(X))$ by \cite{Mich} (respectively, \cite{Banakh}).  As $X$ is an Ascoli space,  $L(X)$ is a subspace of $\CC(\CC(X))$ by  Theorem \ref{t:Ascoli-free-map}. So $L(X)$ is an $\aleph_0$-space  (respectively, a $\Pp_0$-space).
\end{proof}
We do not know whether the condition on $X$ to be an Ascoli space can be omitted as for cosmic spaces in \cite{GM}, namely: Does there exist a non-Ascoli $\aleph_0$-space (a non-Ascoli $\Pp_0$-space) $X$ such that $L(X)$ is an $\aleph_0$-space (respectively, a $\Pp_0$-space)?


\section{The Ascoli property for direct sums of locally convex spaces}   \label{sec-2} 


Let us recall some basic notions used in what follows. We denote by $e$ the unit of a group $G$.

For a non-empty family $\{ G_i \}_{i\in I}$  of groups, the \emph{direct sum} of $G_i$ is denoted by
\[
\bigoplus_{i\in I} G_i :=\left\{ (g_i)_{i\in I} \in \prod_{i\in I} G_i : \; g_i = e_i \mbox{ for almost all } i \right\},
\]
and we  denote by $j_k $ the natural embedding of $G_k$ into $\bigoplus_{i\in I} G_i$; that is,
\[
j_k (g)=(g_i)\in \bigoplus_{i\in I} G_i, \mbox{ where } g_i =g  \mbox{ if } i=k \mbox{ and } g_i =e_i  \mbox{ if } i\not= k.
\]
If $\{ G_i \}_{i\in I}$ is  a non-empty family  of topological groups {\it the final group topology} $\mathcal{T}_f$ on $\bigoplus_{i\in I} G_i$  with respect to the family of canonical homomorphisms $j_k : G_k \to \bigoplus_{i\in I} G_i$  is the finest group topology on $\bigoplus_{i\in I} G_i$ such that all $j_k$ are continuous. For the sake of simplicity we shall identify an element $g_k\in G_k$ with its image $j_k(g_k)$ in $\bigoplus_{i\in I} G_i$ and write simply $g_k$.

For an element $v= g_{i_1}+ \cdots + g_{i_n}\in \bigoplus_{i\in I} G_i$ with $g_{i_k}\not= e_{i_k}$ for every $1\leq k\leq n$, we set $\supp(v):=\{ i_1,\dots,i_n\}$. The {\em support} of a subset $A$ of $\bigoplus_{i\in I} G_i$ is the set
\[
\supp(A):= \bigcup_{v\in A} \supp(v).
\]

Let $\{ (G_i, \tau_i) \}_{i\in I}$ be a non-empty  family of  topological groups and $\mathcal{N}(G_i)$ a basis of open neighborhoods   at the identity in  $G_i$, for each $i\in I$.  For each $i\in I$, fix $U_i \in \mathcal{N}(G_i)$ and put
\[
 \prod_{i\in I} U_i :=\left\{ (g_i)_{i\in I} \in \prod_{i\in I} G_i : \; g_i \in U_i \mbox{ for  all } i \in I \right\}.
\]
Then the sets of the form $ \prod_{i\in I} U_i $, where $U_i \in \mathcal{N}(G_i)$ for every $i\in I$, form a basis of open neighborhoods at the identity of a topological group topology {${\mathcal T}_b$} on $\prod_{i\in I} G_i$ that is called  the \emph{box topology}.
Clearly, $\mathcal{T}_b \leq \mathcal{T}_f$  on $\bigoplus_{i\in I} G_i$.  It is well-known that if $I$ is countable, then  $\mathcal{T}_b = \mathcal{T}_f$.

\begin{proposition}\label{p:boxtopology-bounded}
Let $\{ (G_i, \tau_i) \}_{i\in I}$ be a non-empty  family of  topological groups and $\tau$ be a group topology on $\bigoplus_{i\in I} G_i$ such that $\mathcal{T}_b\leq \tau \leq \mathcal{T}_f$.  If $A$ is a functionally bounded subset of $(\bigoplus_{i\in I} G_i, \tau)$, then $\supp(A)$ is finite.
\end{proposition}
\begin{proof}
Suppose for a contradiction that $\supp(A)$ is infinite. Take a one-to-one sequence $\{ i_n\}_{n\in\NN}$ in $\supp(A)$. Then the projection of $\tau$ on $\bigoplus_{n\in \NN} G_{i_n}$ is $\mathcal{T}_b$. 
Clearly, the projection $B$ of $A$ in the group $G:= \big(\bigoplus_{n\in \NN} G_{i_n}, \mathcal{T}_b\big)$ is also functionally bounded and $\supp(B)=\NN$. Choose a sequence $\{ b_k=(g_{i_n}^{k}): k\in\NN\} $ in $B$ such that for every $k\in \NN$ there is an index $i_{n_{k+1}} \in \supp(b_{k+1})$ such that
\begin{equation} \label{equ:Coproduct-1}
i_{n_{k+1}} \not\in \bigcup_{j\leq k} \supp(b_j).
\end{equation}
Clearly, the sequence  $\{ b_k\}$ is also functionally bounded in $G$. For every $k\in\NN$, take a symmetric neighborhood $U_{i_{n_k}}$ of the identity in $G_{i_{n_k}}$ such that $g_{i_{n_k}}^{k} \not\in U_{i_{n_k}}\cdot U_{i_{n_k}}$. If $i_n\not\in \{ i_{n_1},i_{n_2},\dots\}$, we set $U_{i_n} = G_{i_n}$. Set $U:= G\cap\prod_{n\in\NN} U_{i_n} \in \mathcal{T}_b$.  Now, if $k<m$, then $b_k U\cap b_m U=\emptyset$ since, otherwise, for some $h,t\in U_{i_{n_m}}$ we would have
\[
g_{i_{n_m}}^{m}t= g_{i_{n_m}}^{k}h =h \mbox{ by (\ref{equ:Coproduct-1}), and hence } g_{i_{n_m}}^{m}=ht^{-1}\in U_{i_{n_m}}\cdot U_{i_{n_m}}
\]
that contradicts the choice of $g_{i_{n_m}}^{m}$ and $U_{i_{n_m}}$. So, if $V\in \mathcal{T}_b$ is such that $\overline{V}\overline{V}\subseteq U$, then $\{ b_k V\}_{k\in\NN}$ is a discrete family in $G$. Thus the sequence $\{ b_k\}$ is not functionally bounded in $G$, a contradiction.
\end{proof}

Let $\kappa$ be an infinite cardinal, $\mathbb{V}_\kappa =\bigoplus_{i\in\kappa} \mathbb{R}_i$ be a vector space of dimension $\kappa$ over $\RR$, $\pmb{\tau}_\kappa$ be the box topology on $\mathbb{V}_\kappa$, $\pmb{\mu}_\kappa$  and $\pmb{\nu}_\kappa$  be the maximal and maximal locally convex vector topologies on $\VV_\kappa$ respectively. Clearly, $\pmb{\tau}_\kappa \subseteq \pmb{\nu}_\kappa \subseteq \pmb{\mu}_\kappa$ and $L(D)\cong (\mathbb{V}_\kappa, \pmb{\nu}_\kappa)$, where $D$ is a discrete space of cardinality $\kappa$. It is well-known that $\pmb{\tau}_\omega =\pmb{\nu}_\omega = \pmb{\mu}_\omega$ (see \cite[Proposition 4.1.4]{Jar}). However, if $\kappa$ is uncountable, then  (see \cite{Prot} or  \cite[Theorem 2.1]{Gabr})
\[
\pmb{\tau}_\kappa \subsetneq \pmb{\nu}_\kappa \subsetneq \pmb{\mu}_\kappa.
\]

We shall use the following simple description of the topology $\pmb{\mu}_\kappa$ given in the proof of Theorem 1 in \cite{Prot}.  For each $i\in\kappa$, choose some $\lambda_i \in \RR_i^+, \lambda_i >0$, and denote by $\mathcal{S}_\kappa$ the family of all subsets of $\VV_\kappa$ of the form
\[
\bigcup_{i\in\kappa} \left( [-\lambda_i,\lambda_i ] \times \prod_{j\in\kappa ,\ j\not= i} \{ 0\} \right).
\]
For every sequence $\{ S_k\}_{k\in\omega}$ in $\mathcal{S}_\kappa$, we put
\[
\sum_{k\in\omega} S_k := \bigcup_{k\in\omega} (S_0 + S_1 +\cdots + S_k),
\]
and denote by $\mathcal{N}_\kappa$ the family of all subsets of $\VV_\kappa$ of the form $\sum_{k\in\omega} S_k$. It is easy to check that $\mathcal{N}_\kappa$ is a base at zero $\mathbf{0}$ for $\pmb{\mu}_\kappa$ and the family $\mathcal{\widehat{N}}_\kappa :=\{ \conv(V): V\in \mathcal{N}_\kappa\}$ of convex hulls is a base  at $\mathbf{0}$ for $\pmb{\nu}_\kappa$ (see \cite{Prot}). 

We shall use also the following proposition to show that a space is not Ascoli.
\begin{proposition}[\cite{GKP}] \label{p:Ascoli-sufficient}
Assume   $X$ admits a  family $\U =\{ U_i : i\in I\}$ of open subsets of $X$, a subset $A=\{ a_i : i\in I\} \subset X$ and a point $z\in X$ such that
\begin{enumerate}
\item[{\rm (i)}] $a_i\in U_i$ for every $i\in I$;
\item[{\rm (ii)}] $\big|\{ i\in I: C\cap U_i\not=\emptyset \}\big| <\infty$  for each compact subset $C$ of $X$;
\item[{\rm (iii)}] $z$ is a cluster point of $A$.
\end{enumerate}
Then $X$ is not an Ascoli space.
\end{proposition}

Let us recall some definitions. Let $\mathcal{N}$ be a family  of subsets of a topological space $X$. Then:

$\bullet$ (\cite{Gao}) $\Nn$ is a {\em $cs^\ast$-network at  a point} $x\in X$ if for each sequence $(x_n)_{n\in\w}$ in $X$ converging to  $x$ and for each neighborhood $O_x$ of $x$ there is a set $N\in\mathcal{N}$ such that $x\in N\subset O_x$ and the set $\{n\in\w :x_n\in N\}$ is infinite;

$\bullet$ (\cite{GK-GMS1})   a {\em $cn$-network}  at a point $x\in X$ if for each neighborhood $O_x$ of $x$ the set $\bigcup \{ N \in\Nn : x\in N \subseteq O_x \}$ is a neighborhood of $x$.

We say that $X$ has {\em countable $cs^\ast$-character} ({\em countable $cn$-character}, respectively) if $X$ has  a countable $cs^\ast$-network ($cn$-network, respectively) at each point $x\in X$. It is easy to see that the property to have countable $cs^\ast$-character (or countable $cn$-character) is hereditary.

\begin{theorem} \label{tFreeSpace-Ascoli}
Let $\kappa$ be an uncountable cardinal and let $\tau$ be a vector topology on $\VV_\kappa$ such that $\pmb{\tau}_\kappa \subseteq \tau \subseteq \pmb{\nu}_\kappa$. Then:
\begin{enumerate}
\item[{\rm (i)}] $(\VV_\kappa, \tau)$ is not an Ascoli space;
\item[{\rm (ii)}] {\em (\cite{Gabr})} $(\VV_\kappa, \tau)$  has uncountable tightness;
\item[{\rm (iii)}] $(\VV_\kappa, \tau)$ has uncountable $cs^\ast$-character and uncountable $cn$-character.
\end{enumerate}
In particular, $L(D)$ is not Ascoli for every uncountable discrete space $D$.
\end{theorem}

\begin{proof}
(i) For every $n\in\NN$, set
\[
R_n =\left\{ \xxx=(x_i)\in \VV_\kappa : |\supp(\xxx)|= n \mbox{ and } x_i=\frac{1}{n^2} \mbox{ for every } i\in\supp(\xxx)\right\}
\]
and $R=\bigcup_{n\in\NN} R_n$. Clearly, $\mathbf{0}\not\in R$.

We claim that $\mathbf{0}\in \overline{R}^{\,\pmb{\nu}_\kappa}$ and hence $\mathbf{0}\in\overline{R}^{\, \tau}$.
Take arbitrarily an open convex neighborhood $W$ of $\mathbf{0}$ in $\pmb{\nu}_\kappa$. Choose a neighborhood $\sum_{k\in\omega} S_k$ of $\mathbf{0}$ in $\pmb{\mu}_\kappa$ such that $\sum_{k\in\omega} S_k \subseteq W$. Since $\kappa$ is uncountable, there is a positive number $c>0$ and an uncountable set $J$ of indices such that $\lambda_j^0 > c$ for all $j\in J$, where the positive numbers $\lambda_j^0 $ define $S_0$. Take $n\in\NN$ with $1/n <c$ and a finite subset $J_0=\{ j_1, \dots, j_n\}$ of $J$. For every $1\leq l\leq n$ we set
\[
\mathbf{x}_l =(x_i^l)_{i\in\kappa}, \ \mbox{ where }  x_i^l = \frac{1}{n} \mbox{ if } i=j_l, \mbox{ and } x_i^l = 0 \mbox{ otherwise}.
\]
So $\mathbf{x}_l \in S_0 \subset \sum_{n\in\omega} S_n \subseteq W$ for every $1\leq l\leq n$. Since $W$ is convex the element
\[
\mathbf{x} := \frac{1}{n} (\mathbf{x}_1 +\cdots + \mathbf{x}_n)
\]
belongs to $W$. By construction, $\mathbf{x}\in R_n$. Thus $\mathbf{0}\in \overline{R}^{\,\pmb{\nu}_\kappa}$ and the claim is proven.

For every $n\in\NN$, set
\[
W_n := \mathbb{V}_\kappa \cap \prod_{i\in \kappa} \left( -\frac{1}{10n^3}, \frac{1}{10n^3}\right) \in \pmb{\tau}_\kappa \subseteq\tau.
\]
Now for every $\xxx\in R_n$, set $U(\xxx) := \xxx+W_n$.  Since
\begin{equation} \label{equ:Ascoli-free}
\frac{1}{n^2}-\frac{1}{10n^3} > \frac{1}{(n+1)^2} + \frac{1}{10(n+1)^3} >0,
\end{equation}
we note that $U(\xxx)\cap U(\yyy)=\emptyset$ for every distinct $\xxx,\yyy\in R$.

To prove that the space $(\mathbb{V}_\kappa, \pmb{\nu}_\kappa)$ is not Ascoli it is enough to show that the families $R$ and $\{ U(\xxx) : \xxx\in R\}$ satisfy conditions (i)--(iii) of Proposition \ref{p:Ascoli-sufficient} with $z=\mathbf{0}$. Clearly, (i) holds and (iii) is true by the claim. Let us check (ii).

Let $K$ be a compact subset of $(\mathbb{V}_\kappa, \tau)$. Then, by Proposition \ref{p:boxtopology-bounded}, there is a finite subfamily $F=\{ i_1,\dots,i_n\}$ of $\kappa$ such that $\supp(K)\subseteq F$. If $\supp(\xxx)\nsubseteq F$, (\ref{equ:Ascoli-free}) implies that $U(\xxx)\cap K=\emptyset$. On the other hand, it is easily seen that the number of $\xxx\in R$ such that  $\supp(\xxx)\subseteq F$ is finite. This means that (ii) of Proposition \ref{p:Ascoli-sufficient} also holds true. Thus $(\mathbb{V}_\kappa, \tau)$ is not Ascoli.

(iii) Let $\{ A_n\}_{n\in\w}$ be a sequence of nonempty subsets of $(\mathbb{V}_\kappa, \tau)$. Set
\[
J_0:=\{ n\in\w: |\supp(A_n)|\leq \w\}, J_1:=\w\setminus J_0 \mbox{ and } I:= \kappa\setminus \bigcup_{n\in J_0} \supp(A_n),
\]
so $I$ is uncountable. Assume that $J_1$ is nonempty and let $n_0<n_1<\dots$ be an enumeration of $J_1$. Fix arbitrarily $\alpha_0 \in I$ such that there is $\xxx_0 =(x^0_\alpha)\in A_{n_0}$ for which $c_0 :=\big| x^0_{\alpha_0}\big| >0$. By induction, for every $0<s<|J_1|$, choose $\alpha_s \in I\setminus\{ \alpha_0,\dots,\alpha_{s-1}\}$ for which there is $\xxx_s =(x^s_\alpha)\in A_{n_s}$ such that $c_s :=\big| x^s_{\alpha_s}\big| >0$. Set
\[
U:= \prod_{s<|J_1|} \left( -c_s,c_s\right) \times \prod_{i\in \kappa\setminus\{ \alpha_s:s<|J_1|\} } \IR_i.
\]
Clearly, $U\in \pmb{\tau}_\kappa \subseteq \tau$ and $\xxx_s\not\in U$ for every $s<|J_1|$. Therefore, if  $A_n \subseteq U$, then $n\in J_0$.

(a) Assume for a contradiction that $(\mathbb{V}_\kappa, \tau)$ has countable $cs^\ast$-character. Take a $cs^\ast$-network $\{ A_n\}_{n\in\w}$ at zero. Since $\kappa$ is uncountable it is clear that $J_1$ must be nonempty. Define a convergent sequence $\mathfrak{s}:=\{ \yyy_n =(y^n_\alpha)\}_{n\in\w}$ in $(\mathbb{V}_\kappa, \tau)$ as follows:
\[
y^n_\alpha := \frac{1}{n+1} \mbox{ if } \alpha=\alpha_0, \mbox{ and } y^n_\alpha :=0 \mbox{  otherwise}.
\]
Clearly, $\yyy_n\to \mathbf{0}$ in $\tau$. If $A_n \subseteq U$, then $n\in J_0$, and therefore $\mathfrak{s}\cap A_n=\emptyset$ since $\alpha_0\in I$. Thus $\{ A_n\}_{n\in\w}$ is not a $cs^\ast$-network at zero, a contradiction.

(b) Suppose for a contradiction that $(\mathbb{V}_\kappa, \tau)$ has countable $cn$-character and let $\{ A_n\}_{n\in\w}$ be a $cn$-network at zero. Then the set
\[
B:=\bigcup \{ A_n : \mathbf{0}\in A_n \subseteq U \} = \bigcup_{n\in J_0} A_n,
\]
is not a neighborhood of zero because $\supp(B)$ is countable. Thus $\{ A_n\}_{n\in\w}$ is not a $cn$-network at zero. This contradiction shows that the $cn$-character of $(\mathbb{V}_\kappa, \tau)$ is uncountable.
\end{proof}

\begin{remark} {\em
Let $\{ G_i\}_{i\in I}$ be an uncountable family of topological groups such that uncountable many of them have a nontrivial convergent sequence (or non-discrete) and let $(G,\tau)$ be the direct sum of this family, where $\tau=\mathcal{T}_f$ or $\tau=\mathcal{T}_b$. A similar proof to (iii) of Theorem \ref{tFreeSpace-Ascoli} shows that $(G,\tau)$ has uncountable $cs^\ast$-character (uncountable $cn$-character, respectively). }
\end{remark}

We do not know whether the space $(\mathbb{V}_\kappa, \tau)$ in Theorem \ref{tFreeSpace-Ascoli} is normal.

In what follows we shall use repeatedly the following standard fact, see \cite[Theorem 7.3.5]{NaB}.
\begin{proposition} \label{p:Finite-direct-summund}
For every finite-dimensional subspace $L$ of a locally convex space $E$ there is a closed linear subspace $H$ such that $E=L\oplus H$.
\end{proposition}

Below we prove Theorem \ref{t:Ascoli-directsum-LCS}.
\begin{proof}[Proof of Theorem \ref{t:Ascoli-directsum-LCS}]
For each $i\in \kappa$, by Proposition \ref{p:Finite-direct-summund}, represent $E_i$ in the form $E_i = \IR \oplus \widetilde{E}_i$, where $\widetilde{E}_i$ is a closed subspace of $E_i$. Then
\[
(E,\tau)= (\VV_\kappa, \tau|_{\VV_\kappa}) \times (\widetilde{E}, \tau|_{\widetilde{E}}), \; \mbox{ where } \widetilde{E}:=\bigoplus_{i\in\kappa} \widetilde{E}_i.
\]
As the  direct summand  $(\VV_\kappa, \tau|_{\VV_\kappa})$  of $(E,\tau)$ is not Ascoli by Theorem \ref{tFreeSpace-Ascoli}, the space $(E,\tau)$ is not an Ascoli space by Proposition 5.2 of \cite{BG}. (ii) and (iii) follow from Theorem \ref{tFreeSpace-Ascoli}.
\end{proof}


\section{Proofs of Theorems  \ref{t:Ascoli-c0-bar}, \ref{t:Ck-bar-weakly-Ascoli} and \ref{t:Ascoli-weak-dual}} \label{sec-4}



Recall that a subset $A$ of a locally convex space $E$ is called {\em bounded} if for every neighborhood $U$ of zero there is $\lambda >0$ such that $A\subseteq \lambda U$. We denote by $A^\circ$ the polar of a subset $A$ of $E$. If $(X,Y)$ is a dual pair of vector spaces and $L$ is a linear subspace of $X$ we set $L^\perp :=\{ y\in Y: y(x)=0 \,\forall x\in L\}$. We denote by $X_w$ the space $X$ endowed with the weak topology $\sigma(X,X')$ and set $X'_{w^\ast}:=(X',\sigma(X',X))$.  
Below we prove  Theorem \ref{t:Ascoli-c0-bar}.

\begin{proof}[Proof of  Theorem \ref{t:Ascoli-c0-bar}]
(i) Suppose for a contradiction that there exists an infinite dimensional $\sigma(X',X)$-bounded subset of $X'$. Then $X'$  has an independent and $\sigma(X',X)$-bounded sequence $\{ y_n\}_{n\in\w}$. Clearly, $(1/n)y_n \to 0$ in $\sigma(X',X)$. So the set $K:= \{ 0\}\cup \{ (1/n)y_n\}_{n\in\w}$ is compact in $\sigma(X',X)$. Let us show that $(1/n^2)y_n \to 0$ in $\CC\big(X_w\big)$. Indeed, fix a standard neighborhood $[C,\e]$ of zero in $\CC\big(X_w\big)$, where $C$ is a $\sigma(X,X')$-compact subset of $X$, $\e>0$ and
\[
[C,\e]:=\big\{ f\in \CC\big(X_w\big): |f(x)|<\e \mbox{ for every } x\in C\big\} .
\]
Since $X$ is $c_0$-barrelled, $K$ is equicontinuous, and hence the polar $K^\circ$ of $K$ is a neighborhood of zero in $X$. 
As $C$ is weakly compact it is bounded in $X$. So there is $m>0$ such that $C\subseteq m K^\circ$, and hence $K\subseteq (1/m)C^\circ \subseteq (2/m\e)[C,\e]$.  Since $[C,\e]$ is absolutely convex we obtain
\[
\frac{1}{n^2} y_n= \frac{1}{n}\left( \frac{1}{n} y_n\right) \in [C,\e] \; \mbox{ for every } n> \frac{m\e}{2}.
\]
Thus $(1/n^2)y_n \to 0$ in $\CC\big(X_w\big)$. Since $X_w$ is Ascoli to get a contradiction it is sufficient to show that the compact set  $\KK:=\{ 0\}\cup \{ (1/n^2)y_n\}_{n\in\w} \subseteq \CC\big(X_w\big)$ is not equicontinuous at zero, see  Lemma \ref{l:Ascoli-free-evenly}(iv).

Let $U$ be  a basic neighborhood of zero $0\in X_w$. So there are $\delta>0$ and $z_1,\dots,z_n\in X'$ such that
\[
U=\{ x\in X : \, |z_i(x)|< \delta \mbox{ for } i=1,\dots,n \}.
\]
Denote by $L$ the span of the vectors $z_1,\dots,z_n$ in $X'$. Then $L$ is a $\sigma(X',X)$-closed finite-dimensional subspace of $X'$. Hence there is a closed subspace $H$ of $X'_{w^\ast}:=(X',\sigma(X',X))$ such that $X'_{w^\ast} = L\oplus H$, see Proposition \ref{p:Finite-direct-summund}. Since $L$ is finite dimensional and the $y_n$ are independent, there is $n$ such that $(1/n^2)y_n \not\in L$. As $L=L^{\perp\perp}$ we obtain that $(1/n^2)y_n(x)=a \not= 0$ for some $x\in L^\perp \subseteq X$. Finally, since $(1/a)x\in L^\perp \subseteq U$ and $(1/n^2)y_n \big( (1/a)x\big)=1$ we obtain that $\KK$ is not equicontinuous at $0$. This contradiction shows that every $\sigma(X',X)$-bounded subset of $X'$ must be finite dimensional.

(ii) To prove that $\tau=\sigma(X,X')$ we have to show that every $\tau$-neighborhood $U$ of zero in $X$ is also a weak neighborhood of zero. Note that $\tau$ is the polar topology determined by equicontinuous subsets of $X'$ by \cite[Theorem 8.6.6]{NaB}, and every equicontinuous subset of $X'$ is $\sigma(X',X)$-bounded by \cite[Theorem 8.6.5]{NaB}. So $U$ contains a neighborhood of zero of the form $A^\circ$, where $A$ is a $\sigma(X',X)$-bounded subset of $X'$. We proved in (i) that $A$ is finite dimensional. Hence, as in the previous paragraph, there is a finite dimensional subspace $L_A$ of $X'$, a $\sigma(X',X)$-closed subspace $H_A$ of $X'$ and a standard compact neighborhood $W_A =\prod_{1\leq i\leq \dim(L_A)} [-a_i,a_i] $ of zero in $L_A$ such that
\[
(X',\sigma(X',X))=L_A\oplus H_A \mbox{ and } A\subseteq W_A \times \{ 0\}.
\]
Set $ M:= H_A^\perp, G:=L_A^\perp$, so $X=M\oplus G$. Put  $n:=\dim(L_A)$, $B:=W_A \times \{ 0\}$ and define $F=\{ z_1,\dots,z_n\} \subseteq X'$ by
\[
z_i (m+g):= \frac{1}{a_i} m_i, \mbox{ where } m=(m_1,\dots,m_n)\in M \mbox{ and } g\in G.
\]
Then $F^{\circ}=B^\circ \subseteq A^\circ \subseteq U$. Thus $\tau =\sigma(X,X')$.

(iii) It is well known that for every locally convex space $E$ the space $E_w$ is linearly homeomorphic to a dense subspace of $\IR^\Gamma$, where $\Gamma$ is a Hamel base of $E'$.  As $X=X_w$ by (ii), the assertion follows.
\end{proof}

\begin{proof}[Proof of Corollary \ref{c:Ascoli-dual-Frechet}]
If $X$ is weakly Ascoli, the completness of $X$ and Theorem  \ref{t:Ascoli-c0-bar} imply that $X=\IR^\Gamma$, where $\Gamma$ is a Hamel base of $X'$. Conversely, if $X=\IR^\Gamma $, then $X=X_w$ and $X$ is even a $k_\IR$-space  by \cite{Nob}.
\end{proof}

\begin{example} {\em
The space $\phi$ is barrelled  and complete. Clearly, $\phi \not= \IR^\Gamma$ for every set $\Gamma$. Therefore the space $\phi_w$ is not Ascoli by Corollary \ref{c:Ascoli-dual-Frechet}.}
\end{example}

\begin{corollary} \label{c:weak-Ascoli-closure}
Let $(X,\tau)$ be a locally convex space such that its completion $(\overline{X},\overline{\tau})$ is $c_0$-barrelled. If $(X,\sigma(X,X'))$ is Ascoli, then  $\tau=\sigma(X,X')$ and $(\overline{X}, \sigma(\overline{X},X'))$ is an Ascoli space.
\end{corollary}
\begin{proof}
Note that $\overline{X}'=X'$. So $(X,\sigma(X,X'))$ is a dense subspace of $(\overline{X},\sigma(\overline{X},X'))$, and hence $(\overline{X},\sigma(\overline{X},X'))$  is an Ascoli space by Lemma 2.7 of \cite{GGKZ}. Therefore the topology $\overline{\tau}$ of $\overline{X}$ coincides with $\sigma(\overline{X},X')$ by Theorem  \ref{t:Ascoli-c0-bar}. Thus $\tau=\overline{\tau}|_X  =\sigma(X,X')$.
\end{proof}

Corollaries \ref{c:Ascoli-dual-Frechet} and \ref{c:weak-Ascoli-closure} easily implies the following result proved in \cite[Theorem 6.1.1]{Banakh-Survey}.
\begin{corollary}[\cite{Banakh-Survey}] \label{c:Ascoli-dual-Banach-Taras}
A metrizable locally convex space $(X,\tau)$ is weakly Ascoli if and only if it is weakly metrizable.
\end{corollary}

\begin{proof}
Note that the closure $\overline{X}$ of $X$ is a Fr\'{e}chet space and $\overline{X}'=X'$. Assume that $X_w$ is Ascoli. Then $\overline{X}_w$ is an Ascoli space by Corollary \ref{c:weak-Ascoli-closure}. Therefore, by Corollary \ref{c:Ascoli-dual-Frechet}, $\overline{X}_w =\overline{X} =\IR^N$ for some $N\leq\w$. Thus $X=X_w$ and $X_w$ is metrizable. The converse assertion follows from the Ascoli theorem, see \cite{Eng}.
\end{proof}

To prove Corollary \ref{c:weak-bar-c0-bar} we characterize below  locally convex spaces which are weakly barreled. To this end we need the following easy lemma. 

\begin{lemma} \label{l:funct-boun-boun}
Every functionally bounded subset of  a locally convex space $E$ is bounded.
\end{lemma}
\begin{proof}
Suppose for a contradiction that there is a  functionally bounded subset $A$ of $E$ which is not bounded. So there is an open absolutely convex neighborhood $U$ of $0\in E$ such that $A \nsubseteq\lambda U$ for every $\lambda>0$. Set $\lambda_0 =1$ and choose $a_0\in A$ such that $a_0\not\in 2\lambda_0 U$. Take $\lambda_1>2\lambda_0$ such that $a_0\in \lambda_1 U$. Choose $a_1\in A$ such that $a_1 \not\in 2\lambda_1 U$, and take $\lambda_2 >2\lambda_1 $ such that $a_1 \in \lambda_2 U$. Continuing this process we find a sequence $\{ a_n\}_{n\in\w}$ in $A$ and a sequence $\{ \lambda_n\}_{n\in\w}$ of positive numbers such that
\[
a_n \in \lambda_{n+1} U\setminus 2\lambda_n U \; \mbox{ and } 2 \leq 2\lambda_n < \lambda_{n+1}, \quad n\in \w.
\]
To get a contradiction it is sufficient to show that $\{ a_n\}_{n\in\w}$ is $U$-uniformly discrete, namely, if $ n<m$, then $(a_n+U) \cap (a_m +U)=\emptyset$. Indeed, if $a_n+u=a_m+v$, then $\lambda_{n+1}> 2$ and
\[
a_m=a_n+u-v\in (2+\lambda_{n+1})U\subseteq 2\lambda_{n+1} U
\]
that contradicts the choice of $a_m$.
\end{proof}

\begin{theorem} \label{t:weak-barrelled}
For a locally convex space $E$ the following assertions are equivalent:
\begin{enumerate}
\item[{\rm (i)}] $E_w$ is barrelled;
\item[{\rm (ii)}]  every bounded subset of $E'_{w^\ast}$ is finite-dimensional;
\item[{\rm (iii)}] every  functionally  bounded subset of $E'_{w^\ast}$ is finite-dimensional.
\end{enumerate}
\end{theorem}

\begin{proof}
(i)$\Rightarrow$(ii) Suppose for a contradiction that $E'_{w^\ast}$ has an infinite-dimensional bounded subset. Then $E'$ has an infinite independent weak${}^\ast$ bounded subset $A$. As $E_w$ is barrelled, $A$ is equicontinuous by Theorem 11.3.4 of \cite{NaB}, and hence the polar $A^\circ$ of $A$ is a weak neighborhood of zero in $E$. So there is a finite $F\subseteq E'$  such that $F^\circ \subseteq A^\circ$. Since $A$ is infinite there exists $a\in A$ such that $\spn\big\{ F\cup\{a\}\big\}= \spn\{ F\} \oplus \spn\{ a\}$. By Proposition \ref{p:Finite-direct-summund}, there is a closed subspace $M$ of $E'_{w^\ast}$ such that
\[
E'_{w^\ast} = \spn\{ F\} \oplus \spn\{ a\} \oplus M.
\]
As $(E'_{w^\ast})' =E$, there exists $x\in E$ such that $a(x)\not= 0$ and $\chi(x)=0$ for every $\chi\in \spn\{ F\} \oplus M$. Clearly, $\spn\{ x\} \subseteq F^\circ$. However $\spn\{ x\} \nsubseteq A^\circ$ because the set $\{ a(x):a\in A\}$ is a bounded subset of $\IR$.  Hence $F^\circ \nsubseteq A^\circ$, a contradiction.

(ii)$\Rightarrow$(i) Let $B$ be a bounded subset of $E'_{w^\ast}$. So $B$ is contained in a finite-dimensional subspace $L$ of $E'_{w^\ast}$. By  Proposition \ref{p:Finite-direct-summund}, there is a closed subspace $M$ of $E'_{w^\ast}$ such that $E'_{w^\ast}=L\oplus M$. From this and the fact $(E'_{w^\ast})'=E$ it easily follows that $B$ is equicontinuous. Therefore $E_w$ (and $E$ as well) is barreled by the Banach--Steinhaus theorem.


(ii)$\Rightarrow$(iii) follows from Lemma \ref{l:funct-boun-boun}. Let us prove (iii)$\Rightarrow$(ii). Suppose for a contradiction that there is an infinite  bounded independent subset $A$  of $E'_{w^\ast}$. Choose  a one-to-one sequence $\{ a_n\}_{n\in\w}$ in $A$. Then $\frac{1}{n} a_n$ converges to zero in $E'_{w^\ast}$, so $\{ \frac{1}{n} a_n\}_{n\in\w}\cup\{ 0\}$ is a compact and hence  functionally bounded subset of $E'_{w^\ast}$. Therefore $\{ \frac{1}{n} a_n\}_{n\in\w}$ is an infinite independent functionally bounded subset of  $E'_{w^\ast}$, a contradiction.
\end{proof}

\begin{proof}[Proof of Corollary \ref{c:weak-bar-c0-bar}]
It is sufficient to show that every $c_0$-barrelled weakly Ascoli space $X$ is barrelled. By Theorem \ref{t:Ascoli-c0-bar}, $X=X_w$ and every bounded subset of $X_w$ is finite dimensional. Therefore $X=X_w$ is barrelled by Theorem \ref{t:weak-barrelled}.
\end{proof}


Now we consider the next question: When a barrelled space $\CC(X)$ is weakly Ascoli? We need two assertions.

\begin{lemma} \label{l:Ascoli-weak-dual}
Let $\{ y_n\}_{n\in\w}$ be an independent sequence in a locally convex space $E$. Then for every finite subset $\{ z_0,\dots,z_m\}$ of $E'$ there are $a_0,\dots,a_{m+1}\in \IR$ such that
\[
0\not= a_0 y_0+\cdots +a_{m+1} y_{m+1} \in \bigcap_{i=0}^m \ker(z_i).
\]
\end{lemma}

\begin{proof}
Consider the map $T:\IR^{m+2} \to \IR^{m+1}$ defined by
\[
T(a_0,\dots,a_{m+1}):= A\cdot (a_0,\dots,a_{m+1}), 
\]
where $A$ is the matrix $A:=\big(z_i(y_k)\big)_{i,k}$, $0\leq i\leq m$, $0\leq k\leq m+1$. Since $\ker(T)\not= 0$ there are $a_0,\dots,a_{m+1}\in \IR$ such that
\[
(a_0,\dots,a_{m+1})\in \ker(T)\setminus \{ 0\}.
\]
Then the vector $v:=a_0 y_0+\cdots +a_{m+1} y_{m+1}$ is as desired.
\end{proof}

The following proposition is probably known, but hard to find explicitly stated. So, for the sake of completeness, we give its complete proof. 
\begin{proposition} \label{p:Ascoli-Cp-Ck}
Let $X$ be a Tychonoff space. Then the compact-open topology $\tau_k$ on $C(X)$ coincides with the weak topology $\tau_w$ of $\CC(X)$ if and only if every compact subset of $X$ is finite. In this case $\tau_k=\tau_p=\tau_w$.
\end{proposition}

\begin{proof}
Let $\tau_k=\tau_w$. Recall that the dual space of $\CC(X)$ is the space $M_c(X)$ of all regular Borel measures on $X$ with compact support, see \cite{Jar}. Now suppose for a contradiction that there is an infinite compact subset $K$ of $X$. To get a contradiction it is sufficient to show that the $\tau_k$-neighborhood $[K,1]$ of $0\in C(X)$ does not contain a $\tau_w$-neighborhood of zero. Since $K$ is infinite, there is an infinite discrete sequence $\{ x_n\}_{n\in\w}$ with pairwise disjoint neighborhoods $V_n$ of $x_n$ in $K$, see \cite[Lemma 11.7.1]{Jar}. For every $n\in\w$ choose a function $\widetilde{f}_n: K \to [0,1]$ with support in $V_n$ and $\widetilde{f}_n(x_n)=1$. We apply the Urysohn extension theorem \cite[3.11(c)]{GiJ} to find a continuous function $f_n:X\to [0,1]$ such that $f_n |_K =\widetilde{f}_n$. Clearly, the functions $f_n$, as well as $\widetilde{f}_n$, are independent. Now let $U$ be an arbitrary standard weak neighborhood of zero in $\CC(X)$. So there is $\e>0$ and a finite family of measures $\mu_0,\dots,\mu_m$ with compact support  such that
\[
U=\{ g\in C(X): |\mu_i(g)|<\e \; \forall i=0,\dots,m\}.
\]
By  Lemma \ref{l:Ascoli-weak-dual}, there are $a_0,\dots,a_{m+1}\in \IR$ such that
\[
0\not= h:=a_0 f_0+\cdots +a_{m+1} f_{m+1} \in \bigcap_{i=0}^m \ker(\mu_i).
\]
Since $h(x_i)=a_if_i(x_i)=a_i$ for every $i=0,\dots,m+1$, we obtain that $h$ has nonzero $C(K)$-norm, set $N:=\| h\|_K >0$. As $(2/N)h \in U$ but $(2/N)h\not\in [K,1]$, it follows that $U\nsubseteq [K,1]$. Therefore  $\tau_k$ is strictly finer than $\tau_w$, a contradiction. Thus every compact subset of $X$ is finite.

Conversely, if every compact subset of $X$ is finite, then every measure in $M_c(X)$ has finite support. Thus $\tau_k=\tau_p=\tau_w$.
\end{proof}

Now we prove Theorem \ref{t:Ck-bar-weakly-Ascoli}.

\begin{proof}[Proof of Theorem \ref{t:Ck-bar-weakly-Ascoli}]
(i)$\Rightarrow$(ii) Theorem \ref{t:Ascoli-c0-bar} implies that $\tau_k=\tau_w$. So $\tau_k=\tau_p=\tau_w$ by Proposition \ref{p:Ascoli-Cp-Ck}. Thus $C_p(X)$ is a barrelled Ascoli space.

(ii)$\Rightarrow$(i) Since $C_p(X)$ is barrelled, every functionally bounded subset of $X$ is finite by the Buchwalter--Schmets theorem. So $\tau_k=\tau_p=\tau_w$ and the assertion follows.
\end{proof}

\begin{proof}[Proof of Corollary \ref{c:Ck-weakly-Acoli}]
Since $X$ is a $\mu$-space, $\CC(X)$ is barrelled by the Nachbin--Shirota theorem. So, if $\CC(X)$ is weakly Ascoli, then $C_p(X)$ is barrelled by Theorem \ref{t:Ck-bar-weakly-Ascoli}.  Now the Buchwalter--Schmets theorem implies that all functionally bounded subsets of $X$ are finite. As $X$ is a $k$-space, this implies that $X$ is discrete. Conversely, if $X$ is discrete, then $\CC(X)=\IR^{X}$ carries the weak topology and is even a $k_\IR$-space by \cite{Nob}.
\end{proof}

The condition to be a $k$-space in Corollary \ref{c:Ck-weakly-Acoli} is essential. Indeed, if $X$ is a countable non-discrete space whose functionally bounded sets are finite, then the barrelled space $\CC(X)=C_p(X)$ is metrizable and hence Ascoli.

Below we prove Theorem \ref{t:Ascoli-weak-dual}.
\begin{proof}[Proof of Theorem \ref{t:Ascoli-weak-dual}]
Suppose for a contradiction that $E$ is infinite dimensional. Therefore there is an independent sequence  $\{ y_n\}_{n\in\w}$ of unit vectors in $E'$. Set $L_n :=\spn\{  y_0,\dots,y_{n}\}\subseteq E'$. By Proposition \ref{p:Finite-direct-summund}, for every $n\in\w$ there is a closed subspace $H_n$ of the space $E'_{w^\ast} := (E', \sigma(E',E))$ such that
\[
E'_{w^\ast} = L_n \oplus H_n.
\]
For every $n\in\w$, let $\pi_n : E'_{w^\ast} \to L_n$ be the continuous projection. As the closed unit ball $B$ of the Banach dual $E'$ is $\sigma(E',E)$-compact, there is $b_n >0$ such that
\[
\pi_n (nB) \subseteq [-b_n,b_n]^{n+1}.
\]
Take a continuous function $g_n: L_n=\IR^{n+1} \to [0,1]$ such that
\[
g_n(x)=0 \mbox{ if } x\in [-b_n,b_n]^{n+1}, \mbox{ and } g_n(x)=1 \mbox{ if } x\not\in [-b_n-1,b_n+1]^{n+1}.
\]
Finally we set $f_n := g_n \circ \pi_n, n\in \w$. To get a contradiction we show that: (1) $f_n\to 0$ in $\CC(E'_{w^\ast})$, and (2) the sequence $\{ f_n\}$ is not equicontinuous at $0\in E'_{w^\ast}$.

(1) Let $K$ be a compact subset of $E'_{w^\ast}$. By the Banach--Steinhaus theorem there is $m\in\w$ such that $K\subseteq mB$. Now if $n\geq m$ we obtain that $f_n |_K =0$. So $f_n\to 0$ in the compact-open topology.

(2) Let $U$ be a standard open neighborhood of zero in $E'_{w^\ast}$. So there are $\delta>0$ and $z_0,\dots,z_n\in E$ such that
\[
U=\{ y\in E' : |y(z_i)|<\delta \; \forall i=0,\dots, m\}.
\]
Since $\{ y_n\}_{n\in\w}$ is independent we apply Lemma \ref{l:Ascoli-weak-dual} to find $a_0,\dots,a_{m+1}\in \IR$ such that
\begin{equation} \label{equ:Ascoli-weak-star}
0\not= v:= a_0 y_0+\cdots +a_{m+1} y_{m+1} \in \bigcap_{i=0}^m \ker(z_i).
\end{equation}
Choose $\lambda>0$ such that $\lambda v \not\in [-b_{m+1} -1, b_{m+1} +1]^{m+2}$. Then $f_{m+1}(\lambda v)=1$. Since $\lambda v\in U$ by (\ref{equ:Ascoli-weak-star}), we obtain $|f_{m+1}(\lambda v)-f_{m+1}(0)|=1$, and hence  $\{ f_n\}$ is not equicontinuous.

Now (1) and (2) show that $E'_{w^\ast}$ is not Ascoli, a contradiction. Thus $E$ is finite-dimensional.
\end{proof}

\bibliographystyle{amsplain}

\end{document}